\documentclass{elsarticle}
\usepackage{lineno,hyperref}
\journal{Elsevier}
\usepackage{amsmath,amssymb,amsthm}
\newtheorem{defn}{Definition}[section]

\newtheorem{remark}{Remark}[section]
\newtheorem{thm}{Theorem}[section]

\begin{document}

\begin{frontmatter}

\title{Algebraic properties of bi$-$periodic dual Fibonacci quaternions}

\author{Fatma Ate\c{s}\corref{mycorrespondingauthor}}
\cortext[mycorrespondingauthor]{Corresponding author}
\ead{fgokcelik@ankara.edu.tr}

\author{Ismail G\"{o}k\corref{}}
\ead{igok@science.ankara.edu.tr}

\author{Nejat Ekmekci\corref{}}
\ead{igok@science.ankara.edu.tr}

\address{Ankara University, Faculty of Science, Department of Mathematics, Tandogan, TR06100, Ankara, Turkey}

\begin{abstract}
The purpose of the paper is to construct a new representation of dual
quaternions called bi$-$periodic dual Fibonacci quaternions. These
quaternions are originated as a generalization of the known quaternions in
literature such as dual Fibonacci quaternions, dual Pell quaternions and
dual $k-$Fibonacci quaternions. Furthermore, some of them have not been
introduced until this time. Then, we give generating function, Binet formula
and Catalan's identity in terms of these quaternions.

\end{abstract}


\end{frontmatter}

\section{Introduction}

In linear algebra the complex numbers extend the real numbers by adjoining
one new element $\mathbf{i}$ with the property $\mathbf{i}^{2}=-1$ and the
dual numbers extend the real numbers by adjoining one new nonzero element $%
\varepsilon $ with the property $\varepsilon ^{2}=0.$ Every complex number
has $z=a+\mathbf{i}b$ and the dual number has the form $d=a+\varepsilon b$.

The real quaternions are the number system in $4-$ dimensional vector space
which is extended by the complex numbers. William Rowan Hamilton in 1843
first introduced the number system. Then Clifford generalized the
quaternions to bi$-$quaternions with his work in 1873 and provided to
literature a useful tool for the analysis complex number. They are important
number system which use in kinematics, robotic technology, spatial rotation,
quantum physics and mechanics shortly both theoretical and applied
mathematics or physics.

Generally, there has three types of quaternions called real, complex and dual quaternions. A real quaternion is defined as $ q=q_{0}+q_{1}i+q_{2}j+q_{3}k$ with four units $\left( 1,i, j, k\right) $ where $\left( i, j, k\right) $ are orthogonal unit spatial vectors and $q_{i}$ $(i=0,1,2,3)$ are real numbers. A complex quaternion is an extension of complex numbers if the elements $q_{i}$ $
(i=0,1,2,3)$ are complex numbers and the other a dual quaternion is constructed by the dual numbers $q_{i}$ $(i=0,1,2,3).$ The family of dual quaternions is a Clifford algebra that can be used to represent spatial rigid body displacement in mathematics. Rigid motions in $3-$dimensional
space can be represented by dual quaternions of unit length this fact is used a theoretical kinematic and in applications to $3-$dimensional computer graphics, robotics and computer vision (\cite{car}, \cite{tor}, \cite
{yang}).

Fibonacci sequence was the outcome of a mathematical problem about
well-known rabbit breeding. In algebra, Fibonacci sequence is perhaps one of
the most useful sequence and it has important applications to different
disciplines such as fractal geometry, computer graphics, biology and CAD$-$CAM applications. Fibonacci numbers are characterized by the fact that every number after the first two numbers is the sum of the two proceeding numbers, that is, the numbers have the form
\begin{center}
$ 0, 1, 1, 2, 3, 5, 8, 13, 21, 34, 55,\ldots $ 
\end{center}

To find the $n$ $th$ terms of Fibonacci sequence, Binet's formula which was found by mathematician Jacques Philippe Marie Binet is used explicit formula. Then, there are well known two identities for Fibonacci sequence. In 1680, Jean-Dominique Cassini was discovered Cassini's identity for
Fibonacci numbers. In 1879, Eugene Charles Catalan found the generalization of the Cassini's identity after him called Catalan's identity.

The Fibonacci quaternion is first expressed by Horadam \cite{Horadam} as\\
$Q_{n}=F_{n}+F_{n}i+F_{n+2}j+F_{n+3}k$ where $F_{n}$ is the $n$ $th$
Fibonacci number$.$ After the study, some results in this field have been investigated many authors (\cite{cat}, \cite{cimen}, \cite{falcon}, \cite{hamilton}, \cite{hor}, \cite{iyer}, \cite{iyer2}, \cite{swam}, \cite{yayen}). S. Hal\i c\i\ \cite{Halici} introduced the generating functions and Binet formulas for Fibonacci and Lucas quaternions and established some identities for these quaternions. C. B. \c{C}imen and A. \.{I}pek \cite{cimen} investigated the Pell and Pell$-$Lucas quaternions. In \cite{Tan}, they defined the bi$-$periodic Fibonacci quaternions. Furthermore, they gave the generalization of the Fibonacci quaternions, Pell quaternions, $k-$Fibonacci quaternions. Flaut and Shpakivskyi \cite{flaut} investigated some properties of generalized Fibonacci quaternions and Fibonacci-Narayana quaternions in quaternion algebra. In \cite{kaya2}, the authors investigated the geometric interpretation of the Fibonacci numbers and Fibonacci vectors.

In the present paper, we define the bi$-$periodic dual Fibonacci quaternions. We expressed the generating function, Binet formula, Catalan's identity and Cassini's identity for these quaternions. Hence, this work is a generalization of some quaternions known as dual Fibonacci, Pell and $k-$Fibonacci quaternions.

\section{Basic Concepts and Arguments}
In the present days, many mathematicians studied several problems and
results involving the Fibonacci sequences, and consequently we need the
generalizations of this sequence.

The $n$ $th$ term of Fibonacci number is given by the formula 
\begin{center}
$ F_{n}=F_{n-1}+F_{n-2}, \   \ n \geq2 $
\end{center}
where $F_{0}=0$ and $F_{1}=1$.

Edson and Yayenie \cite{edson} introduced the new generalization of the
Fibonacci numbers called as bi$-$periodic Fibonacci numbers 
\begin{equation}\label{1}
F_{n}=\left\{\begin{array}{lll}
aF_{n-1}+F_{n-2}, \textrm{ if n is even}& \medskip \\
bF_{n-1}+F_{n-2}, \textrm{ if n is odd}& \medskip \\
\end{array}\right.
\end{equation}
with the initial values $F_{0}=0$ and $F_{1}=1$, where $a$ and $b$ are nonzero numbers. In \ref{1}, the classical Fibonacci sequence is obtained if $a=b=1$, the Pell sequence is get if $a=b=2$ and if we take $a=b=k$ , we get the $k-$Fibonacci sequence. Also, the terms of the bi$-$periodic Fibonacci sequence satisfy the equation $F_{n}=(ab+2)F_{n-2}-F_{n-4},$ $n\geq 4.$

The generating function of the Fibonacci sequence is defined following as
\begin{center}
$ F(x)=\frac{x+ax^{2}-x^{3}}{1-(ab+2)x^{2}+x^{4}}. $
\end{center}
The Binet formula for the Fibonacci numbers theory is represented in \cite{stak}. Also, the Binet formula allows us to express the bi$-$periodic dual Fibonacci quaternion in function. The Binet formula of the Fibonacci sequence is denoted by 
\begin{equation}
F_{n}=\frac{a^{\xi (n+1)}}{(ab)^{\left\lfloor \frac{n}{2}\right\rfloor }}
\frac{\alpha ^{n}-\beta ^{n}}{\alpha -\beta }  \label{2}
\end{equation}
where $\xi (n)=n-2\left\lfloor \frac{n}{2}\right\rfloor $ ,$\ $namely, $\xi
(n)=0$ when $n$ is even and $\xi (n)=1$ when $n$ is odd. It can be shown
that $F_{-n}=(-1)^{n-1}F_{n}.$

The 19th century Irish mathematician and physicist William Rowan Hamilton
was impressed by the role of complex numbers in two dimensional geometry.
Hence, he discovered a $4-$dimensional split algebra called the quaternions
and the quaternions introduced by him  
\begin{center}
$ q=a+bi+cj+dk $
\end{center}
where $a$, $b$, $c$, $d$ $  \in R $ and the vectors $i$, $j$ and $k$ are satisfied the following properties: 
\begin{center}
$ i^{2}=j^{2}=k^{2}=ijk=-1, $
\end{center}
\begin{center}
$ ij=k=-ji,jk=i=-kj,ki=j=-ik. $
\end{center}

After him, the dual quaternions are defined a Clifford algebra that can be
used to represent spatial rigid body displacements in mathematics and
mechanics. A dual quaternion is constructed by eight real parameters as
\begin{center}
$ Q=q+\varepsilon q^{\ast } $
\end{center}
where $q$, $q^{\ast }$ are the real quaternions and $\varepsilon $ is the
dual unit such that $\varepsilon \neq 0$ and $\varepsilon ^{2}=0.$ 

Horadam \cite{Horadam} put forward a new definition by taking the Fibonacci numbers replaced with the real numbers in the quaternions. Then, the authors \cite{Kaya} presented the $n$ $th$ dual Fibonacci quaternions $\widetilde{Q}_{n}=Q_{n}+\varepsilon Q_{n+1}$ using the $n$ $th$ dual Fibonacci numbers  $\widetilde{F}_{n}=F_{n}+\varepsilon F_{n+1}$.

\begin{defn}
	The bi$-$periodic Fibonacci quaternions are 
	\begin{center}
	$ Q_{n}=q_{n}+q_{n+1}i+q_{n+2}j+q_{n+3}k $
	\end{center}
	where $q_{n}$ is the $n$ $th$ bi$-$periodic Fibonacci number (\textit{see for details in} \cite{Tan}).
\end{defn}

\section{bi$-$periodic Dual Fibonacci Quatenions}

In this section, we give the new definitions and generalizations of the bi$-$periodic dual Fibonacci quaternions.

\begin{defn}
	The bi$-$periodic dual Fibonacci numbers $\widetilde{F}_{n}$ is 
	\begin{center}
	$ \widetilde{F}_{n}=F_{n}+\varepsilon F_{n+1} $
	\end{center}
	where $F_{n}$ , $F_{n+1}$ are the $n \ {th}$ and ($n+1) \ {th}$ bi$-$periodic
	Fibonacci numbers in \ref{1}, respectively.
\end{defn}

\begin{defn}
	The bi$-$periodic dual $n \ {th}$ Fibonacci quaternions $\widetilde{Q}_{n}$ are defined by
\begin{equation}\label{3}
\widetilde{Q}_{n}= Q_{n}+\varepsilon Q_{n+1}
 \end{equation}
where $i$ , $j$ and $k$ are the standard orthonormal basis in $R^{3}.$
\end{defn}
Let $\widetilde{Q}_{n}=Q_{n}+\varepsilon Q_{n+1}$ and $\widetilde{P}_{n}=P_{n}+\varepsilon P_{n+1}$ be two bi$-$periodic dual Fibonacci quaternions and we define the addition and subtraction as follows:
$$
\widetilde{Q}_{n}\widetilde{P}_{n}=(Q_{n}+P_{n})+\varepsilon
(Q_{n+1}+P_{n+1})$$
here $\widetilde{Q}_{n}=\widetilde{F}_{n}+\widetilde{F}_{n+1}i+\widetilde{F}
_{n+2}j+\widetilde{F}_{n+3}k$ and $\widetilde{P}_{n}=\widetilde{G}_{n}+\widetilde{G}_{n+1}i+\widetilde{G}_{n+2}j+\widetilde{G}_{n+3}k$. Also, the
multiplication of them is given by

$$\widetilde{Q}_{n}\widetilde{P}_{n}=Q_{n}P_{n}+\varepsilon
(Q_{n+1}P_{n}+Q_{n}P_{n+1}).$$

\begin{thm}[Binet Formula] The Binet formula for the bi$-$periodic dual
	Fibonacci quaternion $\widetilde{Q}_{n}$ is given as:\\
	\begin{equation}\label{4}
\widetilde{Q}_{n}=\left\{\begin{array}{lll}
\frac{1}{(ab)^{\left\lfloor \frac{n}{2}\right\rfloor }}\left( \frac{\alpha
	^{\ast }\alpha ^{n}-\beta ^{\ast }\beta ^{n}}{(\alpha -\beta )}\right)  \\ 
+\varepsilon \left( \frac{1}{(ab)^{\left\lfloor \frac{n+1}{2}\right\rfloor }}
\left(\frac{\alpha ^{\ast \ast }\alpha ^{n+1}-\beta ^{\ast \ast }\beta
	^{n+1}}{\alpha -\beta }\right) \right) 
\textit{, if n is even}& \\ 
\\ 

\frac{1}{(ab)^{\left\lfloor \frac{n}{2}\right\rfloor }}\left( \frac{\alpha
	^{\ast \ast }\alpha ^{n}-\beta ^{\ast \ast }\beta ^{n}}{\alpha -\beta }
\right)  \\ 
+\varepsilon \left( \frac{1}{(ab)^{\left\lfloor \frac{n+1}{2}\right\rfloor }}
\left( \frac{\alpha ^{\ast }\alpha ^{n+1}-\beta ^{\ast }\beta ^{n+1}}{
	\alpha -\beta }\right) \right) \textit{, if n is odd}&
\end{array}\right.
\end{equation}
	
	where $\alpha ^{\ast }:=a+\alpha i+\dfrac{a}{ab}\alpha^2 j+\dfrac{1}{ab}\alpha^3 k,$ $\beta
	^{\ast }:=a+\beta i+\dfrac{a}{ab}\beta^2 j+\dfrac{1}{ab}\beta^3 k$ and $\alpha ^{\ast \ast
	}:=1+\dfrac{a}{ab}\alpha i+\dfrac{1}{ab}\alpha^2 j+\dfrac{a}{(ab)^2}\alpha^3 k,$ $\beta ^{\ast \ast
	}:=1+\dfrac{a}{ab}\beta i+\dfrac{1}{ab}\beta^2 j+\dfrac{a}{(ab)^2}\beta^3 k$ .
\end{thm}

\begin{proof}
	Induction method can be used to prove that the Binet formula in equation \ref{4}. First, we show that the statement holds for $ n = 0 $.
	\begin{center}
		 $ \widetilde{Q}_{0}=(0,1,a,ab+1)+\varepsilon(1,a,ab+1,a(ab+2)) $
	\end{center}
and from the Binet formula, we have 
\begin{align*}
 \widetilde{Q}_{0}&=\dfrac{\alpha
	^{\ast }-\beta ^{\ast }}{\alpha -\beta }+\varepsilon(\dfrac{\alpha
	^{\ast \ast }\alpha-\beta ^{\ast \ast }\beta }{\alpha -\beta})\\
  &=\dfrac{1}{\alpha -\beta}(0,\alpha -\beta,\dfrac{a}{ab}(\alpha^2 -\beta^2),\dfrac{1}{ab}(\alpha^3 -\beta^3))+\\
  &\varepsilon \dfrac{1}{\alpha -\beta}(\alpha -\beta,\dfrac{a}{ab}(\alpha^2 -\beta^2),\dfrac{1}{ab}(\alpha^3 -\beta^3),\dfrac{a}{(ab)^2}(\alpha^4 -\beta^4)).
\end{align*}
By simplifying the above equation using the equalities $ \alpha -\beta=-ab $, $ \alpha+\beta=ab $, $ \alpha^2+\beta^2=ab(ab+2) $ , we get
 \begin{center}
 	$ \widetilde{Q}_{0}=Q_{0}+\varepsilon Q_{1} $
 \end{center}
Hence, the Binet formula in \ref{4} satisfied for $ n=0 $.

Assume that the $ \widetilde{Q}_{n} $ holds. It must then be shown that $ \widetilde{Q}_{n+1} $ holds. In here, there are two situations, $ n $ is even or odd.

If we take $ n $ is even, it is satisfied the following equation
\begin{align*}
\widetilde{Q}_{n}=
\frac{1}{(ab)^{\left\lfloor \frac{n}{2}\right\rfloor }}\left( \frac{\alpha
	^{\ast }\alpha ^{n}-\beta ^{\ast }\beta ^{n}}{(\alpha -\beta )}\right)+\varepsilon \left( \frac{1}{(ab)^{\left\lfloor \frac{n+1}{2}\right\rfloor }}
\left(\frac{\alpha ^{\ast \ast }\alpha ^{n+1}-\beta ^{\ast \ast }\beta
	^{n+1}}{\alpha -\beta }\right) \right).
\end{align*}
For  $ n+1 $,
\begin{align*}
\widetilde{Q}_{n+1}=Q_{n+2}+\varepsilon Q_{n+3}
\end{align*}
 By using Theorem 2 in the \cite{Tan}, we can write
  \begin{align*}
 Q_{n+2}=\frac{1}{(ab)^{\left\lfloor \frac{n+2}{2}\right\rfloor }}\left( \frac{\alpha
 	^{\ast }\alpha ^{n+2}-\beta ^{\ast }\beta ^{n+2}}{(\alpha -\beta )}\right)
 \end{align*} 
 and 
 \begin{align*}
 Q_{n+3}=\frac{1}{(ab)^{\left\lfloor \frac{n+3}{2}\right\rfloor }}\left( \frac{\alpha
 	^{\ast \ast }\alpha ^{n+3}-\beta ^{\ast \ast}\beta ^{n+3}}{(\alpha -\beta )}\right)
 \end{align*}
thereby showing that indeed $ \widetilde{Q}_{n+1} $ holds.
So, the Binet formula is obtained for all even numbers and using the same method it is shown for all odd numbers.
\end{proof}

\begin{remark}
	If we take $a=b=1$ in the above Binet formula of the bi$-$periodic dual Fibonacci quaternions, then we have the Binet formula of the dual Fibonacci quaternions in \cite{Kaya}.
\end{remark}

\begin{thm}[Generating function] The generating function for the bi-periodic
	dual Fibonacci quaternions is
	\begin{equation}
	G(t)=\frac{Q_{0}+(Q_{1}-bQ_{0})t+(a-b)R(t)}{1-bt-t^{2}}+\varepsilon \left( 
	\frac{Q_{1}+(Q_{2}-bQ_{1})t+(a-b)S(t)}{1-bt-t^{2}}\right)  \label{5}
	\end{equation}
	where $R(t)=tf(t)+(f(t)-t)i+(\frac{f(t)}{t}-1)j+(\frac{f(t)}{t^{2}}-\frac{1}{t}-(ab+1)t)k$ and $S(t)=(f(t)-t)+(\frac{f(t)}{t}-1)i+\left( \frac{f(t)}{t^{2}}-\frac{1}{t}-(ab+1)t\right) j+(\frac{f(t)}{t^{3}}-\frac{1}{t^{2}}-(ab+1))k$ with $f(t):=\sum\limits_{k=1}^{\infty }F_{2k-1}t^{2k-1}.$
\end{thm}

\begin{proof}
	The ordinary generating function of the bi$-$periodic dual Fibonacci
	quaternions is $G(t)=\sum\limits_{n=0}^{\infty }\widetilde{Q}_{_{n}}t^{n}$
	and using the equations $btG(t)$ and $t^{2}G(t)$,
	\begin{align*}
	G(t)=&\widetilde{Q}_{0}+\widetilde{Q}_{_{1}}t+\widetilde{Q}
	_{_{2}}t^{2}+\ldots +\widetilde{Q}_{_{n}}t^{n}+\ldots  \\
	btG(t) =&b\widetilde{Q}_{0}t+b\widetilde{Q}_{_{1}}t^{2}+b\widetilde{Q}
	_{_{2}}t^{3}+\ldots +b\widetilde{Q}_{_{n}}t^{n+1}+\ldots  \\
	t^{2}G(t) =&\widetilde{Q}_{0}t^{2}+\widetilde{Q}_{_{1}}t^{3}+\widetilde{Q}
	_{_{2}}t^{4}+\ldots +\widetilde{Q}_{_{n}}t^{n+2}+\ldots 
	\end{align*}
	and take into account the equation $\widetilde{Q}_{n}=Q_{n}+\varepsilon
	Q_{n+1}$ in the above equations, we get
	\begin{align*}
	G(t)-btG(t)-t^{2}G(t) &=Q_{0}+\varepsilon
	Q_{_{1}}+[(Q_{_{1}}-bQ_{0})+\varepsilon (Q_{2}-bQ_{_{1}})]t \\
	&+[(Q_{_{2}}-bQ_{_{1}}-Q_{0})+\varepsilon
	(Q_{_{3}}-bQ_{_{2}}-Q_{_{1}})]t^{2} \\
	&+[(Q_{_{3}}-bQ_{_{2}}-Q_{_{1}})+\varepsilon
	(Q_{_{4}}-bQ_{_{3}}-Q_{_{2}})]t^{3}+\ldots  \\
	&+[(Q_{_{n}}-bQ_{_{n-1}}-Q_{_{n-2}})+\varepsilon
	(Q_{n+1}-bQ_{_{n}}-Q_{n-1})]t^{n}+\ldots  \\
	&=Q_{0}+(Q_{_{1}}-bQ_{0})t+\sum\limits_{n=2}^{\infty }\left(
	Q_{n}-bQ_{n-1}-Q_{n-2}\right) t^{n} \\
	&+\varepsilon \left\{
	Q_{_{1}}+(Q_{2}-bQ_{_{1}})t+\sum\limits_{n=2}^{\infty }\left(
	Q_{n+1}-bQ_{n}-Q_{n-1}\right) t^{n}\right\} 
	\end{align*}
	And using $\left( \ref{1}\right) ,$ we can calculate 
	\begin{align*}
	\sum\limits_{n=2}^{\infty }\left( Q_{n}-bQ_{n-1}-Q_{n-2}\right) t^{n}
	=&(a-b)R(t), \\
	\sum\limits_{n=2}^{\infty }\left( Q_{n+1}-bQ_{n}-Q_{n-1}\right) t^{n}
	=&(a-b)S(t).
	\end{align*}
	
	So, the generating function becomes in \ref{5} for the
	bi-periodic dual Fibonacci quaternions.
\end{proof}

\begin{remark}
	The generating function of the dual Fibonacci quaternions for $a=b=1$ is
	obtained as follows 
	\begin{equation*}
	G(t)=\frac{Q_{0}+(Q_{1}-Q_{0})t}{1-t-t^{2}}+\varepsilon \left( \frac{Q_{1}+(Q_{2}-Q_{1})t}{1-t-t^{2}}\right) .
	\end{equation*}
\end{remark}

Now, we give the important identities of the bi-periodic dual Fibonacci
quaternions known as Catalan and Cassini's identities .

\begin{thm}[Catalan's identity] Let$\ r$ be nonnegative even integer number. The Catalan's identity for the bi$-$periodic dual Fibonacci quaternions is given by\\
	
	$\widetilde{Q}_{_{n-r}}\widetilde{Q}_{_{n+r}}-\widetilde{Q}_{_{n}}^{2}=$\\
	
	$\left\{ 
	\begin{array}{c}
	\begin{array}{c}
	\frac{\alpha ^{\ast }\beta ^{\ast }((ab)^{r}-\beta ^{2r})+\beta ^{\ast
		}\alpha ^{\ast }((ab)^{r}-\alpha ^{2r})}{(ab)^{r}(\alpha -\beta )^{2}} \\ 
	\\ 
	+\varepsilon \left[ \frac{(\alpha ^{\ast \ast }\beta ^{\ast }\alpha +\alpha
		^{\ast }\beta ^{\ast \ast }\beta )((ab)^{r}-\beta ^{2r})+(\beta ^{\ast
		}\alpha ^{\ast \ast }\alpha +\beta ^{\ast \ast }\alpha ^{\ast }\beta
		)((ab)^{r}-\alpha ^{2r})}{(ab)^{r}(\alpha -\beta )^{2}}\right]
	\end{array}
	, \text{if n is even}, \\ 
	\\ 
	\\ 
	\begin{array}{c}
	-\frac{\alpha ^{\ast \ast }\beta ^{\ast \ast }((ab)^{r}-\beta ^{2r})+\beta
		^{\ast \ast }\alpha ^{\ast \ast }((ab)^{r}-\alpha ^{2r})}{(ab)^{r-1}(\alpha
		-\beta )^{2}} \\ 
	\\ 
	+\varepsilon \{-\frac{(\alpha ^{\ast }\beta ^{\ast \ast }\alpha +\alpha
		^{\ast \ast }\beta ^{\ast }\beta )((ab)^{r}-\beta ^{2r})+(\beta ^{\ast
		}\alpha ^{\ast \ast }\beta +\beta ^{\ast \ast }\alpha ^{\ast }\alpha
		)((ab)^{r}-\alpha ^{2r})}{(ab)^{r}(\alpha -\beta )^{2}}\}%
	\end{array}
	,\text{if n is odd},
	\end{array}
	\right. $\\
	
	where $n$ and $r$ are the nonnegative integer numbers with $n\geq r.$
\end{thm}

\begin{proof}
	From $\left( \ref{3}\right) $, we can write the following equality$\ $
	\begin{eqnarray*}
		\widetilde{Q}_{_{n-r}}\widetilde{Q}_{_{n+r}}-\widetilde{Q}_{_{n}}^{2}
		&=&(Q_{_{n-r}}+\varepsilon Q_{n-r+1})\times (Q_{n+r}+\varepsilon
		Q_{n+r+1})-(Q_{n}+\varepsilon Q_{n+1})^{2} \\
		&=&Q_{_{n-r}}Q_{_{n+r}}-Q_{_{n}}^{2} \\
		&&+\varepsilon (Q_{n-r+1}Q_{n+r}+Q_{n-r}Q_{n+r+1}-Q_{n}Q_{n+1}-Q_{n+1}Q_{n})
	\end{eqnarray*}
	and use the Binet formula for the bi-periodic dual Fibonacci quaternion in \ref{4}, we can find the desired result.
\end{proof}

\begin{thm}[Cassini's identity]
	$(i)$ Cassini's formula of the bi-periodic dual Fibonacci quaternions for 
	\textbf{odd }indices is 
	\begin{eqnarray*}
		\widetilde{Q}_{_{2m-1}}\widetilde{Q}_{_{2m+3}}-\widetilde{Q}_{_{2m+1}}^{2}
		&=&-\left\{ \frac{((ab)^{2}-\beta ^{4})\alpha ^{\ast \ast }\beta ^{\ast \ast
			}+((ab)^{2}-\alpha ^{4})\beta ^{\ast \ast }\alpha ^{\ast \ast }}{(ab)(\alpha
			-\beta )^{2}}\right\} \\
		&& \\
		&&-\varepsilon \left\{ 
		\begin{array}{c}
			\frac{\alpha \left[ ((ab)^{2}-\beta ^{4})\alpha ^{\ast }\beta ^{\ast \ast
				}+((ab)^{2}-\alpha ^{4})\beta ^{\ast \ast }\alpha ^{\ast }\right] }{%
				(ab)^{2}(\alpha -\beta )^{2}} \\ 
			\\ 
			+\frac{\beta \left[ ((ab)^{2}-\alpha ^{4})\beta ^{\ast }\alpha ^{\ast \ast
				}+((ab)^{2}-\beta ^{4})\alpha ^{\ast \ast }\beta ^{\ast }\right] }{(ab)^{2}(\alpha -\beta )^{2}}
		\end{array}%
		\right\} , \\
		&&
	\end{eqnarray*}
	$(ii)$ For \textbf{even} indices Cassini's formula is \\
		\begin{eqnarray*}
		\widetilde{Q}_{_{2m-2}}\widetilde{Q}_{_{2m+2}}-\widetilde{Q}_{_{2m}}^{2} &=&
		\frac{((ab)^{2}-\beta ^{4})\alpha ^{\ast }\beta ^{\ast }+((ab)^{2}-\alpha
			^{4})\beta ^{\ast }\alpha ^{\ast }}{(ab)^{2}(\alpha -\beta )^{2}} \\
		&& \\
		&&+\varepsilon \left\{ 
		\begin{array}{c}
			\frac{(\alpha ^{\ast \ast }\beta ^{\ast }\alpha +\alpha ^{\ast }\beta ^{\ast
					\ast }\beta )((ab)^{2}-\beta ^{4})}{(ab)^{2}(\alpha -\beta )^{2}} \\ 
			\\ 
			+\frac{(\beta ^{\ast }\alpha ^{\ast \ast }\alpha +\beta ^{\ast \ast }\alpha
				^{\ast }\beta )((ab)^{2}-\alpha ^{4})}{(ab)^{2}(\alpha -\beta )^{2}}
		\end{array}%
		\right\} . \\
		&&
	\end{eqnarray*}
\end{thm}

\begin{proof}
	$(i)$ If we take $r=2$ and $n=2m+1$ in the Catalan's formula then we get the
	Cassini's formula.
	
	$(ii)$ We put the $r=2$ and $n=2m$ in the Catalan's formula, we obtain the
	Cassini's formula for even indices.
\end{proof}

\section{Conclusions}

We want to emphasized that the characterizations some of the dual Fibonacci
quaternions can be introduced with the convenient parameter $a$ and $b.$ For
example, the following results are stated

$(i)$ for dual Fibonacci quaternions with the parameter $a=b=1,$

$(ii)$ for dual Pell quaternions with the parameter $a=b=2$,

$(iii)$ for dual $k-$Fibonacci quaternions with the parameter $a=b=k$.

Although, the option $(i)$ is introduced by Nurkan and G\"{u}ven in \cite{Kaya}, the other options $(ii)$ and $(iii)$ are unnecessary to publish in
literature. Because, these are special cases of our paper.

\end{document}